\documentclass[11pt,a4paper]{amsart}

\usepackage{amsmath}
\usepackage{amsthm}
\usepackage{amssymb}
\usepackage{geometry, constants}
\usepackage{color}
\usepackage{mathabx}				

\usepackage{microtype} 
\usepackage[hidelinks]{hyperref} 
\hypersetup{final} 

\newtheorem{theorem}{Theorem}[section]
\newtheorem{lemma}[theorem]{Lemma}
\newtheorem{corollary}[theorem]{Corollary}
\newtheorem{proposition}[theorem]{Proposition}

\theoremstyle{remark}
\newtheorem{remark}[theorem]{Remark}

\newtheorem{definition}[theorem]{Definition}

\numberwithin{equation}{section}

\def\Ex{\mathbb{E}} 
\def\Pr{\mathbb{P}} 
\newcommand*{\er}{\mathbb{R}} 
\newcommand*{\eps}{\varepsilon} 
\DeclareMathOperator*{\sgn}{sgn} 
\newcommand*{\ind}{{\mathbf 1}} 
\def\maxi{\max_{1\le i \le m}} 
\def\maxj{\max_{1\le j \le n}} 
\def\maxij{\max_{\substack{1\le i \le m \\ 1\le j \le n}}} 
\def\Cov{\operatorname{Cov}} 
\def\Id{\operatorname{Id}} 

\title[Estimates of norms of log-concave random matrices]{Estimates of norms of log-concave random matrices with dependent entries}
\author{Marta Strzelecka}
\date{February 4, 2019}
\address{Institute of Mathematics, University of Warsaw, Banacha 2, 02--097 Warsaw, Poland.}
\email{martast@mimuw.edu.pl}
\thanks{The research was supported by the National Science Centre, Poland via the grants 2015/19/N/ST1/02661 and 2018/28/T/ST1/00001.}
\keywords{Random matrices, operator norm,  log-concave vectors, unconditional vectors.}
\subjclass[2010]{60B20, 46B09, 15B52}

\begin{document}

\maketitle
%
%

\begin{abstract}
 We prove estimates for $\Ex \| X:  \ell_{p'}^n \to \ell_q^m\|$ for $p,q\ge 2$ and any random matrix $X$ having the entries  of the form $a_{ij}Y_{ij}$, where $Y=(Y_{ij})_{1\le i\le m, 1\le j\le n}$ has i.i.d.   isotropic log-concave  rows. This generalises the result of Gu{\'e}don, Hinrichs,  Litvak, and Prochno for Gaussian matrices with independent entries. Our estimate is optimal up to  logarithmic factors. As a byproduct we  provide the analogue bound for $m\times n$ random matrices, which entries form an unconditional vector in $\er^{mn}$.  We also prove bounds for norms of matrices which entries are certain Gaussian mixtures.
 \end{abstract}


\section{Introduction and  main result}
	
A classical result regarding spectra of random matrices is Wigner's Semicircle Law, which describes the limit of empirical spectral measures of a random matrix with independent centred entries with equal variance. Theorems of this type say nothing about the largest eigenvalue (i.e. the operator norm).  However, Seginer proved in \cite{MR1762786} that for a random matrix $X$ with  i.i.d. 
 symmetric entries $\Ex\|X\|_{2,2}$ (by $\|A\|_{p,q}$ we denote the operator norm of the matrix $A$ from $\ell_p$ to $\ell_q$) is of the same order as the expectation of the maximum Euclidean norm of rows and columns of $X$. The same holds true for the structured Gaussian matrices (i.e. when $X_{ij}=a_{ij}g_{ij}$ and $g_{ij}$ are i.i.d. standard Gaussian variables), as was recently shown by Lata{\l}a, van Handel, and Youssef in \cite{MR3878726}, and up to a logarithmic factor for any $X$ with independent centred entries, see \cite{MR3035757}. The advance of the two latest results is that they do not require that the entries of $X$ are equally distributed (nor that they have  equal variances). 

Another   upper bound for $\Ex\|X\|_{2,2}$ also does not require  equal distributions but only the independence of entries: by \cite{MR2111932} we know that 
\[
	\Ex\|X\|_{2,2}\lesssim \max_i \Bigl(\sum_j \Ex X_{ij}^2\Bigr)^{1/2} + \max_j \Bigl(\sum_i \Ex X_{ij}^2\Bigr)^{1/2}+  \Bigl(\sum_{i,j} \Ex X_{ij}^4\Bigr)^{1/4}.
\]
This bound is dimension free, but in some cases is worse than the one from \cite{MR3035757}.
 
Upper bounds for the expectation of other operator norms were investigated in \cite{MR0393085} in the case of independent centred  entries  bounded by $1$. For $q\ge 2$ and $m\times n$ matrices the authors proved that $\Ex\|X\|_{2,q} \lesssim \max\{m^{1/q}, \sqrt n\}$. In \cite{MR3645120} Gu{\'e}don, Hinrichs,  Litvak, and Prochno proved that for a structured Gaussian matrix $X=(a_{ij}X_{ij})_{i\le m, j\le n}$ and $p,q \ge 2$,
		\begin{multline*}
			\Ex \| X\|_{p',q}\le C(p,q) \biggl[ \bigl(\log m \bigr)^{1/q}\maxi \Bigl( \sum_{j=1}^n |a_{ij}|^p \Bigr)^{1/p} + \maxj \Bigl(\sum_{i=1}^m |a_{ij}|^q \Bigr)^{1/q} 
			\\
			+ 
			\bigl(\log m \bigr)^{1/q} \Ex \maxij |X_{ij}| \biggr].
		\end{multline*}
This estimate is optimal up to logarithmic factors (see Remark \ref{matrix_reverse_bound} below). Note that in the case $(p,q)\neq(2,2)$ moment method fails in estimating $\Ex \| X\|_{p' ,q}$ (as it gives information only on the spectrum of $X$).
	
	All the mentioned results require the independence of entries of $X$. 
In this article we will see how to generalise the main result of \cite{MR3645120} to a wide class of random matrices  with independent uncorrelated log-concave rows, following the scheme of proof of the original theorem from \cite{MR3645120}. In order to obtain the key estimates for log-concave vectors needed in the proof we use the comparison of weak and strong moments of $\ell_p$-norm of $X$ from \cite{MR3503729} and a Sudakov minoration-type bound from \cite{MR3274967}.

	Our estimate is  optimal (for fixed $p,q\ge 2$) up to a factor depending logarithmically on the dimension. Let us stress that we do not require the rows of $X$ to have independent, but only uncorrelated coordinates (and to be log-concave) --- we require the independence only between the rows.
	 	
	Before we state our main results, let us say a few words about log-concave vectors. We say that a random vector $X$ in $\er^n$ is log-concave, if for any compact nonempty sets $K,L\subset \er^n$ and $\lambda\in [0,1]$, 
\[
	\Pr\bigl(X\in \lambda K+(1-\lambda)L\bigr)\geq \Pr(X\in K)^{\lambda}\Pr(X\in L)^{1-\lambda}.
\]
The class of log-concave vectors is closed under linear transformations, convolutions and weak limits. 
By the result of Borell \cite{MR0388475} an $n$-dimensional vector with a full dimensional support is log-concave
if and only if it has a log-concave density, i.e. has a density of the form $e^{-h}$, where $h$ is a convex function with values
in $(-\infty,\infty]$. 

Log-concave vectors are a natural generalisation of vectors distributed uniformly  over convex bodies. Moreover, distribution of any log-concave vector can be obtained as a weak limit of projections of uniform measures over (higher dimensional) convex bodies (see for example \cite{MR2220210}).  Other results and conjectures about log-concave vectors are discussed in monograph \cite{MR3185453}. 

We say that a vector $X$ in $\er^n$ is  isotropic if $\Cov X =\Id$. If $X$ is a log-concave random vector in $\er^n$ with full dimensional support, then there exists a linear transformation $T$ such that $\Cov(TX)=\Id$, so the isotropicity is only a matter of normalisation.

	To make the notation more clear, if $A=(A_{ij})_{i \le m, j\le n}$ is an $m\times n$ matrix, we denote by $A_i \in \er^n$  its $i$-th row and by $A^{(j)}\in\er^m $ we denote its $j$-th column. We are ready now to present the main theorem.


	\begin{theorem}\label{main_thm_matrices}
		Let $m\ge 2$, let  $Y_1, \ldots, Y_m$ be i.i.d.  isotropic log-concave vectors in $\er^n$, and let $A=(A_{ij})$ be an $m\times n$ (deterministic) matrix. Consider a random matrix $X$ with entries $X_{ij}=A_{ij}Y_{ij}$ for $i\le m, j\le n$, where $Y_{ij}$ is the $j$-th coordinate of $Y_i$.
		Then for every $p,q \ge 2$ we have
				\begin{align}\label{eq:main_thm_matrices}
			\Ex  \| X\|_{p',q}
			 \le C(p,q) \Bigl[ \bigl(\log m \bigr)^{1/q}\maxi \bigl\| A_{i} \bigr\|_p + \maxj \bigl\| A^{(j)} \bigr\|_q 
			+ 
			\bigl(\log m \bigr)^{1+\frac1q} \Ex \maxij |X_{ij}| \Bigr],
		\end{align}
	where $C(p,q)$ depends only on $p$ and $q$.
	\end{theorem}

	\begin{remark}\label{matrix_reverse_bound}
		Note that the bound from Theorem \ref{main_thm_matrices} is optimal up to a constant depending on $p, q$ and logarithmically on the dimension. Indeed, since $Y_{ij}$ is log-concave we have by the regularity of $Y_{ij}$ (see \eqref{regular_moments} below) that $\Ex|Y_{ij}| \ge (2\Cr{seminorms})^{-1} \big(\Ex Y_{ij}^2\bigr)^{1/2} = (2\Cr{seminorms})^{-1}$.  Hence for every $j\le n$, (we take $u=e_j$, use the unconditionality of $\|\cdot\|_q$ and the Jensen inequality)
		\begin{align*}
		\Ex \|X\|_{p',q} = \Ex \sup_{u\in \ell_{p'}^n} \|Xu\|_q \ge \Ex\|X^{(j)}\|_q = \Ex \bigl\| \bigl(|Y_{ij}|A_{ij}\bigr)_i  \bigr\|_q \ge (2C_1)^{-1} \|A^{(j)} \|_q.
		\end{align*}
		Since $\|X\|_{p',q} = \|X^T\|_{q',p}$, we also have $\Ex \|X\|_{p',q} \ge (2C_1)^{-1} \|A_i \|_p$ for all $i\le m$. Moreover, for all $i\le m$  and $j\le n$, (we take $v=e_i$ and $u=e_j\sgn X_{ij}$)
		\begin{align*}
	  \|X\|_{p',q} = \sup_{u\in \ell_{p'}^n} \sup_{v\in \ell_{q'}^n} v^TXu  \ge |X_{ij}|.
	  		\end{align*}
			Therefore 
			\begin{align*}
			\Ex \| X\|_{p',q} \ge (4\Cr{seminorms}+1)^{-1} \Bigl[\maxi \bigl\| A_{i} \bigr\|_p + \maxj \bigl\| A^{(j)} \bigr\|_q 
			+ 
			\Ex \maxij |X_{ij}| \Bigr],
		\end{align*}
		what yields the claim.
	\end{remark}

	The next corollary is a version of Theorem \ref{main_thm_matrices} in the spirit of the aforementioned results from \cite{MR1762786, MR3878726, MR3035757}. It follows directly from \eqref{eq:main_thm_matrices}, and the Jensen inequality.
	
	\begin{corollary}\label{cor_matrices}
		Under the assumptions of Theorem \ref{main_thm_matrices} we have
		\[
		\Ex \| X\|_{p',q}\le C(p,q) \biggl( (\log m)^{1+\frac 1q} \Ex \maxi \Bigl(\sum_{j=1}^n |X_{ij}|^p \Bigr)^{1/p} + \Ex  \maxj \Bigl(\sum_{i=1}^m |X_{ij}|^q \Bigr)^{1/q}  \biggr).
		\]
	\end{corollary}
	
	\begin{remark}
	If the rows and columns of $Y$ are isotropic and log-concave (we do not require  independence), and $p,q\ge 1$, then 
	\begin{multline} \label{eq_remark_matrices}
		\Ex \maxi \Bigl(\sum_{j=1}^n |A_{ij}Y_{ij}|^p \Bigr)^{1/p} + \Ex  \maxj \Bigl(\sum_{i=1}^m |A_{ij}Y_{ij}|^q \Bigr)^{1/q}
		\\ \le C \Bigl( p^2\maxi \bigl\| A_{i} \bigr\|_p + q^2\maxj \bigl\| A^{(j)} \bigr\|_q   +(p+q) \log (m \vee n) \Ex \maxij |A_{ij}Y_{ij}|\Bigr),
	\end{multline}
	what means that the bound we used in the proof of Corollary \ref{cor_matrices} (the one which uses the Jensen inequality) may be reversed (in the log-concave setting) up to a logarithmic factor and constants depending only on $p$ and $q$. Therefore the estimates from Theorem \ref{main_thm_matrices} and Corollary \ref{cor_matrices} are equivalent up to a logarithmic factor. Inequality \eqref{eq_remark_matrices} follows directly from the following proposition.
	\end{remark}
	
	\begin{proposition} \label{prop_matrices}
		Let $Y$ be an $m\times n$ random matrix,  with isotropic and log-concave rows, let $B$ be a deterministic $m\times n$ matrix, and let $p\ge 1$. Then
		\begin{equation*}
			\Ex \maxi  \Bigl(\sum_{j=1}^n |B_{ij}Y_{ij}|^p \Bigr)^{1/p}  \lesssim p^2 \maxi \Bigl( \sum_{j=1}^n |B_{ij}|^p \Bigr)^{1/p} + p\log (m \vee n) \Ex \maxij |B_{ij}Y_{ij}| .
		\end{equation*}
	\end{proposition}
	
	It turns out that instead of assuming the log-concavity, we may assume the unconditionality, i.e. that an $m\times n$ random matrix we consider, treated as  an $(mn)$-dimensional vector, is unconditional (we no longer assume the independence of rows). Recall that we say that a random vector $Z$ in $\er^d$ is unconditional, if for every choice of signs $\eta\in\{-1,1 \}^d$ the vectors $Z$ and $(\eta_iZ_i)_{i\le d}$ are equally distributed (or, equivalently, that $Z$ and $(\eps_iZ_i)_{i\le d}$ are equally distributed, where $\eps_1, \ldots ,\eps_d$ are i.i.d. symmetric Bernoulli variables, independent of $Z$). The assertion of the next corollary is expressed in the spirit of Corollary \ref{cor_matrices}, which is more natural in the non log-concave setting (without the assumption of log-concavity the assertions of Theorem \ref{main_thm_matrices} and Corollary \ref{cor_matrices} are no longer equivalent).
	\begin{corollary}\label{thm_uncond_matrices}
		Assume that $X$ is a random matrix such that  the $(mn)$-dimensional vector $(X_{1,1},\ldots X_{1,n}, X_{2,1}, \ldots , X_{2,n}, X_{m,1},\ldots, X_{mn})$ is unconditional. Then for every $p,q \ge 2$ we have
				\begin{multline}\label{eq:main_thm_matrices}
			\Ex  \| X\|_{p',q} \le
			 C(p,q) \biggl( (\log m)^{\frac 32+\frac1q}\Ex \maxi \Bigl(\sum_{j=1}^n |X_{ij}|^p \Bigr)^{1/p} 
			 \\+ \sqrt{\log n}\Ex  \maxj \Bigl(\sum_{i=1}^m |X_{ij}|^q \Bigr)^{1/q}  \biggr),
		\end{multline}
	where $C(p,q)$ depends only on $p$ and $q$.
	\end{corollary}

	The rest of this note is organised as follows. Section \ref{sect_preliminaries} contains results from other articles, which will be used in a sequel. Section \ref{sect:matr_proofs} contains generalisations of Lemmas 3.1 and 3.2 from \cite{MR3645120} to the log-concave setting and the proof of  Theorem \ref{main_thm_matrices}.
	In Section \ref{sect:Gmixtures} we will show how to deduce an analogue of Theorem \ref{main_thm_matrices} for Gaussian mixtures (see Corollary \ref{cor:Gmixtures}) and we will provide a proof of Proposition \ref{prop_matrices}. Section \ref{sect_uncond} is devoted to the proof of Corollary \ref{thm_uncond_matrices}.

	{\bf Notation.}  By $C$ we denote universal constants. If a constant $C$ depends on a parameter $\alpha$, we express it as  $C(\alpha)$.  The value of $C, C(\alpha)$ may differ at each occurrence. Whenever we want to fix the value of an absolute constant we use letters $C_1,C_2,\ldots$.
We may always assume that $C,C_i\geq 1$. For two quantities $a,b$ we write $a\lesssim b$ if there exists a constant $C$, such that $a\le Cb$, and $a\sim b$, if $a\lesssim b$ and $b\lesssim a$. For two numbers $a$ and $b$ we write $a\vee b$ instead of $\max\{a,b\}$.

For a random variable $X$ by $\|X\|_p$ we denote the $p$-th integral norm of $X$, i.e. the quantity $(\Ex |X|^p )^{1/p}$ (in the case $X=\|Y\|$ we also call this quantity the $p$-th strong moment of $Y$ associated with the norm $\|\cdot\|$). For a vector $x\in \er^n$ (in particular for a random vector $X$) and $r\ge 1$, by $\|x\|_r$ we denote the $\ell_r$-norm of $x$, i.e. $\|x\|:= (\sum_{i=1}^n |x_i|^r)^{1/r}$. For $r=2$ we shall also write $|\cdot|$ instead of $\|\cdot\|_2$. It will be always clear from the context, what $\|X\|_q$ means for a random object $X$, so the double meaning of $\|\cdot\|_q$ will not lead to any misunderstanding.  Recall that for an $m\times n$ matrix $A$ by $\|A\|_{p,q}$ we denote its norm from $\ell_p^n$ to $\ell_q^m$. For $p\in [1,\infty]$ by $p'$ we denote the H{\"o}lder conjugate of $p$, i.e. $1=\frac 1p+\frac 1{p'}$.

	\section{Preliminaries}\label{sect_preliminaries}
	
	We will frequently use the regularity of $f(Z)$ for log-concave vectors $X$ and seminorms $f$, i.e. 
	\begin{equation} \label{regular_moments}
		(\Ex f(Z)^p)^{1/p}\leq \Cl{seminorms} \frac{p}{q}(\Ex f(Z)^q)^{1/q} \quad\text{for $p\geq q\geq 1$}
	\end{equation}
 (see \cite[Theorem 2.4.6]{MR3185453}).

	We will also need the comparison of weak and strong moments for $\ell_p$-norms of log-concave vectors:
	
	\begin{theorem}[{\cite[Theorem 5]{MR3503729}}]\label{thm:Paourlr}
		Let $Z$ be a log-concave vector in $\er^n$, and let $p\in[1,\infty)$. Then
		\[
			(\Ex\|Z\|_p^q)^{1/q}\leq C p\Bigl(\Ex\|Z\|_p+ \sigma_{p,X}(q)  \Bigr)\quad \mbox{ for }q\geq 1,
		\] 
		where
		\[
			\sigma_{p,X}(q):= \sup_{t\in B_{p'}^n} \Bigl\|\sum_{i=1}^n t_i Z_i \Bigr\|_q 
		\]
		is the $q$-th weak moment of $X$ associated with the $\ell_p$-norm.
	\end{theorem}
	
	We will use the previous theorem also in the tail-bound version:

	\begin{corollary}
	Assume $Z$ is a log-concave vector in $\er^n$, and $p\in [1,\infty)$. Then
	\begin{equation} \label{cond(iv)_log-concave}
		\Pr\Bigl(\|Z\|_p \ge
\Cl{d1}p\bigl(u+\Ex\|Z\|_p \bigr)\Bigr)
\le 
\Cl{d2}\sup_{t\in  B_{p'}^n}\Pr\biggl(\Bigl|\sum_{i=1}^nt_iZ_i\Bigr|\ge u\biggr).
	\end{equation}
	\end{corollary}

	For the Reader's convenience we give a proof of this corollary, which goes along the lines of the proof of Corollary 1.3 in \cite{MR3778221}.

	\begin{proof}
		Define a random variable $S:=\|Z\|_p$.
By the Paley--Zygmund inequality and \eqref{regular_moments} we have for $t\in \er^n$, and $q\ge 1$,
\begin{align}
\notag
\Pr\Biggl(\Bigl|\sum_{i=1}^n t_iZ_i \Bigr|\ge \frac{1}{2}\biggl\|\sum_{i=1}^n t_iZ_i \biggr\|_q\Biggr)
&=\Pr\biggl(\Bigl|\sum_{i=1}^n t_iZ_i \Bigr|^q\ge 2^{-q}\Ex\Bigl|\sum_{i=1}^n t_iZ_i \Bigr|^q\biggr) 
\\
\label{eq:PZbelow}
&\ge 
(1-2^{-q})^2\Biggl(\frac{\bigl\|\sum_{i=1}^n t_iZ_i\bigr\|_q}{\bigl\|\sum_{i=1}^n t_iZ_i\bigr\|_{2q}}\Biggr)^{2q}
\ge e^{-\Cl{c4}q}. 
\end{align}

In order to show \eqref{cond(iv)_log-concave} we consider 3 cases.\\

\textbf{Case 1.} $2u<\sup_{t\in  B_{p'}^n}\|\sum_{i=1}^n t_iZ_i\|_{2}$. Then by \eqref{eq:PZbelow}
\[
\sup_{t\in  B_{p'}^n}\Pr\biggl(\Bigl|\sum_{i=1}^nt_iZ_i\Bigr|\ge u\biggr)
\ge e^{-2\Cr{c4}}
\]
and \eqref{cond(iv)_log-concave} obviously holds if $\Cr{d2}\ge \exp(2\Cr{c4})$. 

\textbf{Case 2.} $\sup_{t\in  B_{p'}^n}\|\sum_{i=1}^n t_iZ_i\|_{2}\le 
2u<\sup_{t\in B_{p'}^n}\|\sum_{i=1}^n t_iZ_i\|_{\infty}$.
Let us then define
\[
q:=\sup\biggl\{r\ge 2\Cr{c4}\colon\ 
\sup_{t\in  B_{p'}^n}\Bigl\|\sum_{i=1}^n t_iZ_i\Bigr\|_{r/\Cr{c4}}\le 2u\biggr\}.
\]
By \eqref{eq:PZbelow} we have

\[
\sup_{t\in  B_{p'}^n}\Pr\biggl(\Bigl|\sum_{i=1}^n t_iZ_i\Bigr|\ge u\biggr)\ge e^{-q}.
\]

By \eqref{regular_moments},  Theorem \ref{thm:Paourlr}, and Chebyshev's inequality
we have
\[
\Pr\bigl(S\ge \Cl{c5}p(\Ex S+u)\bigr)\le
\Pr(S\ge e\|S\|_q)\le e^{-q}
\]
 for $\Cr{c5}$
large enough.
Thus \eqref{cond(iv)_log-concave} holds in this case.

\textbf{Case 3.} $u>\sup_{t\in  B_{p'}^n}\|\sum_{i=1}^n t_iZ_i\|_{\infty}=\|S\|_\infty$. Then
$\Pr(S\ge u)=0$ and \eqref{cond(iv)_log-concave} holds for any $\Cr{d1}\ge 1$.
\end{proof}

	In the proof of   Theorem \ref{main_thm_matrices} we will use Theorem 2.1 from \cite{MR3645120}, which is another version of results provided before by Gu{\'e}don--Rudelson in \cite{MR2304336}, and by Gu{\'e}don--Mendelson--Pajor--Tomczak-Jaegerman in \cite{MR2490210}:
	
	\begin{theorem}[{\cite[Theorem 2.1]{MR3645120}}] \label{thm:2.1Guedon}
		Let $E$ be a Banach space with modulus of convexity of power type $2$ with constant $\lambda$. Let $X_1, \ldots X_m \in E^*$  be independent random vectors, and let $q\ge  2$.  Define
		\begin{equation}\label{eq:defu}
			u:= \sup_{t\in B_E} \Bigl( \sum_{i=1}^m \Ex\bigl| \langle X_i, t \rangle \bigr|^q \Bigr)^{1/q},
		\end{equation}
		and
		\begin{equation}\label{eq:defv}
			v:= \Bigl( \lambda^8 \bigl(T_2(E^*)\bigr)^2 \log m\hspace{0.08cm} \Ex \maxi \|X_i\|_{E^*}^q \Bigr)^{1/q},
		\end{equation}
		where $T_2(E^*)$ is the Rademacher type $2$ constant of $E^*$.
		Then
		\[
			\biggl[\Ex \sup_{t\in B_E} \biggl| \sum_{i=1}^m \Bigl( \bigl| \langle X_i, t\rangle\bigr|^q -  \Ex \bigl| \langle X_i, t\rangle\bigr|^q \Bigr)\biggr| \biggr]^{1/q}\le C (\sqrt{uv} + v) \le 2C(u+v).
		\]
 	\end{theorem}
		
		We will use  Theorem \ref{thm:2.1Guedon} with $E=\ell_{p'}^n$. In this case $\lambda$ and $T_2(E^*)$ are known.

	\section{Proof of Theorem \ref{main_thm_matrices}}\label{sect:matr_proofs}
	

	The next two lemmas provide estimates of the quantities $u$ and $v$ appearing in Theorem \ref{thm:2.1Guedon} in the case $E=B_{p'}^n$.
	
	\begin{lemma}\label{lemma:maxX_i}
		Assume $p, q, X$, and $Y$ are as in Theorem \ref{main_thm_matrices}. Then
		\begin{align*}
			\Bigl( \Ex \maxi \|X_i \|_p^q \Bigr)^{1/q} \le C(p,q)\Bigl[  \maxi \| A_i \|_p + \log m \ \Ex \maxij |X_{ij}|  \Bigr],
		\end{align*}
		where $C(p,q)$ depends only on $p$ and $q$.	
		\end{lemma}
	
	\begin{lemma} \label{lem:moms}
		Assume $p, q, X$, and $Y$ are as in Theorem \ref{main_thm_matrices}. Then 
		\begin{equation} \label{lemma est for sigma}
			\sup_{t\in B_{p'}^n} \biggl( \sum_{i=1}^m \Ex \bigl| \langle X_i, t \rangle \bigr|^q \biggr)^{1/q} \le  \Cr{seminorms} q\maxj \bigl\|A^{(j)}  \bigr\|_q.
		\end{equation}
	\end{lemma}
	
	In the proof of Lemma  \ref{lemma:maxX_i} we will also need the following estimate:
	
	\begin{lemma}\label{lem:sudakov}
		Assume that $Z$ is an isotropic log-concave vector in $\er^m$. Then for all $1\le k \le m$ and all $a\in \er^m$ we have
		\[
			\Ex \maxi |a_iZ_i| \ge \frac 1{\Cl{d3}}\max_{k\le m} \bigr( a_k^* \min_{i\le m} \|Z_i\|_{\log (k+1)}  \bigl),
		\]
		where $(a_i^*)_{i=1}^m$ denotes the non-increasing rearrangement of $(|a_i|)_{i=1}^m$.
	\end{lemma}

		In order to prove Theorem \ref{main_thm_matrices}, we repeat the proof scheme from \cite{MR3645120}.
			
	\begin{proof}[Proof of Theorem \ref{main_thm_matrices}]
		We  use  Theorem \ref{thm:2.1Guedon} for $E=\ell_{p'}^n$. Then $\lambda \sim p$ (see \cite[Theorem 5.3]{MR1999201})  and $T_2(E^*) \sim \sqrt p$. Let $u$ and $v$ be given by formulas \eqref{eq:defu} and \eqref{eq:defv}. 
		The triangle inequality, Theorem \ref{thm:2.1Guedon}, Lemma \ref{lemma:maxX_i}, and Lemma \ref{lem:moms}  yield
		\begin{align*} 
			\Ex\|X\|_{p',q}  & \le \bigl(\Ex\|X\|_{p',q}^q \bigr)^{1/q} = \biggl[ \Ex\sup_{t\in B_{p'}^n }\sum_{i=1}^m \bigl| \langle t, X_i \rangle \bigr|^q  \biggr]^{1/q} 
			\\ 
			&
			\le\biggl[ \Ex \sup_{t\in B_{p'}^n} \biggl| \sum_{i=1}^m \Bigl( \bigl| \langle X_i, t\rangle\bigr|^q -  \Ex \bigl| \langle X_i, t\rangle\bigr|^q \Bigr)\biggr| \biggr]^{1/q} + \sup_{t\in B_{p'}^n } \biggl(\Ex \sum_{i=1}^m \bigl| \langle t, X_i \rangle \bigr|^q \biggr)^{1/q}
			\\ & \le C\cdot(u+v) 
			  \\ &\le C(p,q) \Bigl[ \bigl(\log m \bigr)^{1/q}\maxi \bigl\| A_{i} \bigr\|_p + \maxj \bigl\| A^{(j)} \bigr\|_q 
			+ 
			\bigl(\log m \bigr)^{\frac 1q+1} \Ex \maxij |X_{ij}| \Bigr]. 
					\end{align*}
	\end{proof}

	The main  contribution of this article lies in the proofs of Lemmas  \ref{lemma:maxX_i}, \ref{lem:moms}, and \ref{lem:sudakov}.
			
	\begin{proof}[Proof of Lemma \ref{lem:sudakov}]
		We may and do assume that $a_1\ge a_2 \ge \ldots \ge a_m \ge 0$, i.e. $a_i^*=a_i$ for $i\le m$.
		By \cite[Proposition 3.3]{MR3274967} we have for all $k\le m$,
		\[
			\Ex \max_{1\le i\le k} |a_iZ_i| \ge C^{-1} \min_{1\le i\le k} \|a_iZ_i\|_{\log (k+1)} \ge 
			 C^{-1} a_k \min_{1\le i\le m} \|Z_i\|_{\log (k+1)}.
		\]
		Thus
		\[
			\Ex \max_{1\le i\le m} |a_iZ_i| = \max_{1\le k\le m} \Ex \max_{1\le i\le k} |a_iZ_i| \ge  C^{-1} \max_{1\le k\le m} \bigl(a_k \min_{1\le i\le m} \|Z_i\|_{\log (k+1)}\bigr). \qedhere
		\]
	\end{proof}

	\begin{proof}[Proof of Lemma \ref{lemma:maxX_i}]
		We may and do assume that $m\ge2$.
		
		Since we may approximate $A_{ij}$ by nonzero numbers, we may and do assume that $a_{ij}\neq 0$ for all $i,j$.		Let $\Cr{d1},\Cr{d2}$ be the constants from \eqref{cond(iv)_log-concave}, let $\Cr{d3}$ be the constant from Lemma \ref{lem:sudakov}, and recall that $\Cr{seminorms}$ is the constant from \eqref{regular_moments}. We may assume  that all these constants are greater than $1$.
		
		Note that for any $a,b \in \er$ we have $a=(a-b)_+ + a\wedge b$. Thus, by the triangle inequality,
		\begin{multline}\label{pom12}
			\Bigl( \Ex \maxi \|X_i\|_p^q \Bigr)^{1/q}
			\\ \le \biggl( \Ex \maxi \Bigl[\bigl( \|X_i\|_p - \Cr{d1}p\Ex \|X_i\|_p \bigr)^q \ind_{\{\|X_i\|_p\ge \Cr{d1}p\Ex\|X_i\|_p\}} \Bigr]\biggr)^{1/q} + \Cr{d1}p\maxi \Ex\|X_i\|_p.
		\end{multline}
		 Moreover, for every $1\le i\le m$ we have by \eqref{regular_moments} and the isotropicity of $Y_i$, that
		\begin{align}\label{pom11}\nonumber
			\Ex\|X_i\|_p & \le \Bigl(\sum_{j=1}^n \Ex |Y_{ij}|^p |A_{ij}|^p \Bigr)^{1/p} \le     \max_{j\le n} \|Y_{ij}\|_p \| A_{i} \|_p \le  \Cr{seminorms}p \|A_{i} \|_p  \\ & \le   \Cr{seminorms}p \max_{1\le k \le m}\|A_{k} \|_p.
		\end{align}
		Now we pass to the estimation of the fist term of \eqref{pom12}.
		Let 
		\[
			B:= \Cr{seminorms}^2\Cr{d3}\log (m+1)  \ \Ex\maxij |X_{ij}|
		\qquad \text{and } \quad
			\sigma:= (\maxi \sigma_{p, X_i}(2)) \vee B.
		\]
		 By \eqref{cond(iv)_log-concave} we have
		\begin{align} \label{by_parts_lemma} \nonumber
			 \Ex  \maxi & \Bigl[\bigl( \|X_i\|_p  - \Cr{d1}p\Ex \|X_i\|_p \bigr)^q \ind_{\{\|X_i\|_p\ge \Cr{d1}p\Ex\|X_i\|_p\}} \Bigr]
			 \\  \nonumber &
			 \le (2\Cr{d1}pe\sigma)^q+ \int_{2\Cr{d1}pe\sigma}^\infty qu^{q-1} \Pr \Bigl( \maxi \bigl(\|X_i\|_p - \Cr{d1}p\Ex \|X_i\|_p \bigr) \ge u  \Bigr) du
			 \\ \nonumber &
			 \le (2\Cr{d1}pe\sigma)^q + (\Cr{d1}p)^q \sum_{i=1}^m  \int_{2e\sigma}^\infty qu^{q-1} \Pr \bigl( \|X_i\|_p - \Cr{d1}p\Ex \|X_i\|_p \ge \Cr{d1}pu  \bigr) du
			 \\  &
			 \le(2\Cr{d1}pe\sigma)^q + (\Cr{d1}p)^q\Cr{d2} \sum_{i=1}^m  \int_{2e\sigma}^\infty qu^{q-1}  \sup_{\|t\|_{p'}\le 1 } \Pr\biggl(\Bigl| \sum_{j=1}^n  t_j  X_{ij} \Bigr| \ge u\biggr) du.
		\end{align}
		For $u \ge \sup_{\|t\|_{p'}\le 1}\|\sum_{j=1}^n  t_j  X_{ij} \|_{\infty} $ the function we integrate vanishes, so from now on we will consider only $i$'s for which $u< \sup_{\|t\|_{p'}\le 1}\|\sum_{j=1}^n  t_j  X_{ij} \|_{\infty}$. 
		
		Note that if $1\le i \le m$ and $\sup_{\|t\|_{p'}\le 1}\|\sum_{j=1}^n  t_j  X_{ij} \|_{\infty} > u\ge e \sigma \ge e\sigma_{p, X_i}(2)$, then  
		\[
			r:=r(i):=\sup\{s\ge 2: \sigma_{p, X_i}(s) \le u/e \}\in [2,\infty)
		\]
		 and $\sigma_{_p, X_i}(r) = u/e$. Therefore
		\begin{equation} \label{tail_est_lemma}
			\sup_{\|t\|_{p'}\le 1 } \Pr \biggl(\Bigl| \sum_{j=1}^n  t_j  X_{ij} \Bigr|\ge u\biggr)  \le \frac{\sup_{\|t\|_{p'}\le 1 } \| \langle t, X_i \rangle \|_r^r}{u^r} 			= e^{-r}.
		\end{equation}

		Now we will estimate $r$ from below.
		For $t\ge 2$  let 
		\[
			\varphi(t)=t \min_{1\le j\le n}\|Y_{ij}\|_t.
		\]
		 Since $Y_i's$ are identically distributed, $\varphi$ does not depend on $i$. By \eqref{regular_moments}, and the isotropicity of $Y$ we have
		\begin{align} \label{pom007} \nonumber
			\sigma_{p, X_i}(t) \le \sigma_{2, X_i}(t) & \le \Cr{seminorms}t \max_{|x|\le 1} \biggl( \Ex \Bigl(\sum_{j=1}^n  A_{ij}Y_{ij}x_j \Bigr)^2 \biggr)^{1/2} 
			\\ \nonumber &
			 =  \Cr{seminorms}t \max_{|x|\le 1} \biggl( \Ex \Bigl(\sum_{j=1}^n  A_{ij}^2x_j^2 \Bigr)^2 \biggr)^{1/2} 
			\\ &
			=  \Cr{seminorms}t \maxj |A_{ij}| \cdot \|Y_{ij}\|_2  \le  \Cr{seminorms} \varphi(t) \maxj |A_{ij}|.
		\end{align}

		Since we can permute the rows of $A$, we may and do assume that 
		\begin{equation*}
			\maxj |A_{1j}| \ge \ldots \ge \maxj |A_{mj}|.
		\end{equation*}
		Let $j(i)\le n$ be such an index that $|A_{i j(i)}| = \maxj |A_{ij}|$. 
		Lemma \ref{lem:sudakov} applied to $Z_i=Y_{ij(i)}$ and the non-increasing sequence $a_i=|A_{ij(i)}|$ implies 
		\begin{equation*}
			\Ex \maxij |X_{ij}|\ge \Ex \maxi |A_{ij(i)}Y_{i j(i)}|\ge \Cr{d3}^{-1} \bigl(\log(m+1)\bigr)^{-1} \maxi \Bigl( \varphi \bigl(\log(i+1)\bigr) |A_{ij(i)}| \Bigr),
		\end{equation*}
		so for all $i\le m$ we have 
		\[
			B\ge \Cr{seminorms}^2 \varphi(\log (i+1)) |A_{ij(i)}| = \Cr{seminorms}^2 \varphi(\log (i+1))  \maxj |A_{ij}|.
		\]
				Note that by \eqref{regular_moments} for all $r \ge \lambda \ge 2$ we have $\sigma_{p, X_i}(r/\lambda) \ge \sigma_{p, X_i}(r) /(\Cr{seminorms}\lambda)$. Take $\lambda=\sigma_{p , X_i}(r)/ B =u/(Be) \ge 2$. Then by a calculation similar to the one above we get
		\begin{align*}
			\frac ue = \sigma_{p,X_i}(r) \le \frac{\Cr{seminorms}r}2  \maxj |A_{ij}| \le C_1^2r \maxij |A_{ij}|\Ex|Y_{ij}| \le C_1^2r\Ex\maxij|X_{ij}| \le Br,
		\end{align*} 
				so indeed $r\ge \lambda\ge 2$.
				
				Therefore for all $i\le m$ we have
		\begin{equation} \label{varphi_estimate}
			\frac B{\Cr{seminorms}} =\frac1{ \lambda\Cr{seminorms}} \sigma_{p, X_i}(r) \le \sigma_{p, X_i}(r/\lambda) \mathop{\le}^{\eqref{pom007}}  \Cr{seminorms}\varphi \Bigl(\frac r\lambda \Bigr) \maxj |A_{ij}|\le \frac{B\varphi(\frac r\lambda)}{\Cr{seminorms}\varphi(\log(i+1))}.
		\end{equation}
		Since the function $\varphi$ is strictly increasing, the previous inequality yields $r\ge \lambda \log (i+1)$. This together with \eqref{tail_est_lemma} implies that (recall that $\lambda = \frac{u}{Be} \ge 2$)
%
		\begin{multline} \label{sum_lemma}
			  \sum_{i=1}^m   \sup_{\|t\|_{p'}\le 1 }  \Pr\Bigl(\Bigl|\sum_{j=1}^n  t_j  X_{ij} \Bigr|\ge u\Bigr) 
			  \le \sum_{i=1}^m  (i+1)^{-\frac{u}{eB} } 
			   \le
			   2^{-\frac{u}{eB} } + \int_2^\infty x^{-\frac{u}{eB} } dx \le 3\cdot 2^{-\frac{u}{e\sigma} }.
		\end{multline}
		
		Inequalities \eqref{by_parts_lemma}, \eqref{sum_lemma}, and the Stirling formula yield that
		\begin{equation}\label{pom13}
			\biggl( \Ex \Bigl[ \maxi  \bigl( \|X_i\|_p  - \Cr{d1}\Ex \|X_i\|_p \bigr)^q \ind_{\{\|X_i\|_p\ge \Cr{d1}\Ex\|X_i\|_p\}} \Bigr] \biggr)^{1/q}\le C\Cr{d1}\Cr{d2}^{1/q} \sigma pq.
		\end{equation}
		Moreover, by \eqref{regular_moments} 
		\begin{equation*}
			\maxi \sigma_{p, X_i}(2) \le 2 \Cr{seminorms} \maxi \sigma_{p, X_i}(1)\le 2\Cr{seminorms} \maxi \Ex \|X_i\|_p,
		\end{equation*}
		where the second inequality holds since the weak first moment is bounded above by the strong first moment. This together with \eqref{pom12}, \eqref{pom11}, and \eqref{pom13} gives the assertion.
			\end{proof}

	\begin{proof}[Proof of Lemma \ref{lem:moms}]
	Note that if $0\le r\le s$, then for every $x\in \er^n$ we have $\|x\|_{s}\le \|x\|_r$, so we may and do assume $p=2$. By \eqref{regular_moments}, the isotropicity of $Y$,  and the Jensen inequality we have
		\begin{align*}
			\sup_{t\in B_{2}^n} \biggl( \sum_{i=1}^m \Ex \bigl| \langle X_i, t \rangle \bigr|^q \biggr)^{1/q} &
			\le \Cr{seminorms} q \sup_{\|t\|_2 \le 1}\biggl( \sum_{i=1}^m  \Bigl(\Ex \bigl| \langle X_i, t \rangle \bigr|^2 \Bigr)^{q/2} \biggr)^{1/q} 
			\\ &
			=  \Cr{seminorms} q  \sup_{\|t\|_2 = 1}\biggl( \sum_{i=1}^m  \Bigl( \sum_{j=1}^n A_{ij}^2t_j^2 \Bigr)^{q/2} \biggr)^{1/q} 
			\\ & 
			\le 
			  \Cr{seminorms} q  \sup_{\|t\|_2 = 1}\biggl( \sum_{i=1}^m   \sum_{j=1}^n |A_{ij}|^qt_j^2  \biggr)^{1/q}
			 \\ &
			 =  \Cr{seminorms} q  \biggl( \sup_{\|t\|_2 = 1}    \sum_{j=1}^n \bigl\| A^{(j)}  \bigr\|_q^q  t_j^2  \biggr)^{1/q}
			 \\ &
			  =  \Cr{seminorms} q \maxj \bigl\| A^{(j)}  \bigr\|_q. \qedhere
		\end{align*}
	\end{proof}

	\begin{remark}\label{rem_p^gamma}
		By the same reasoning as in the log-concave case, we may prove (using \cite[Corollary 1.3]{MR3778221}, \cite[Theorem 2.1]{MR3335827}, and the claim below instead of \eqref{cond(iv)_log-concave}, Lemma \ref{lem:sudakov} and the previous estimates on $\sigma_{p,X_i}(s)$, respectively) the following. 
		
		Let $X$ be an $m\times n$ random matrix with entries $X_{ij}=A_{ij}Y_{ij}$, where $Y_{ij}$ are independent symmetric random variables such that $\Ex Y_{ij}^2=1$. Assume that for any $r\ge 2$ and any $1\le i\le m$, $1\le j \le n$ we have $\frac{r^{\beta}}L\le\|Y_{ij}\|_r\le L r^{\beta}$ with $\beta \in [\frac12, 1]$. Then for every $p,q \ge 2$ we have
				\begin{align*}
			\Ex  \| X\|_{p',q}
			\le C(p,q, L) \Bigl[ \bigl(\log m \bigr)^{1/q}\maxi \bigl\| A_{i} \bigr\|_p + \maxj \bigl\| A^{(j)} \bigr\|_q 
			+ 
			\bigl(\log m \bigr)^{1/q} \Ex \maxij |X_{ij}| \Bigr].
		\end{align*}
	where $C(p,q,L)$ depends only on $p$, $q$, and $L$. At the end of Section \ref{sect:Gmixtures} we provide another result concerning this type of random matrices (see Corollary \ref{ex_GM}).
		
		As we mentioned, it suffices to prove the claim:
		\begin{equation} \label{remark comp of sums}
			\biggl\| \sum_{j=1}^n t_jY_{ij} \biggr\|_r \le CL r^{\beta} \biggl\| \sum_{j=1}^n t_jY_{ij} \biggr\|_2= CL r^{\beta} \|t\|_2,
		\end{equation}
		where $C$ is an absolute constant, and repeat the proof of Theorem \ref{main_thm_matrices}.

\begin{proof}[Proof of the claim]
		It suffices to consider  $r=2k$, where $k$ is an integer. Let us denote
		\begin{equation*}
			c_{i_1, \ldots ,i_n}:= {i_1+\ldots +i_n  \choose i_1}{i_2+\ldots + i_n \choose i_2} \ldots {i_n \choose i_n}.
		\end{equation*}
		Let $G=(G_{j})_{j=1}^n$ be the standard $n$-dimensional Gaussian vector. Recall that for any $t\in \er^n$ and $r\ge 1$ we have $\|\sum_{j=1}^n t_j G_j\|_r =\|t\|_2 \|G_1\|_r \sim \|t\|_2 \sqrt{r} =\sqrt{r}\|\sum_{j=1}^n t_j Y_{ij}\|_2  $.
		
		By the assumptions on $Y_{i}$ and by the fact that $\beta \ge \frac 12$  we get
 		\begin{align*}
			\biggl\| \sum_{j=1}^n t_j Y_{ij} \biggr\|_{2k}^{2k} & = \sum_{j_1+\ldots + j_n =k} c_{2j_1, \ldots ,2j_n} \Ex Y_{i1}^{2j_1} \cdots \Ex Y_{in}^{2j_n} t_1^{2j_1} \cdots t_n^{2j_n}
			\\ & \le
			 L^{2k} \sum_{j_1+\ldots + j_n =k} c_{2j_1, \ldots ,2j_n} (2j_1)^{2 j_1 \beta} \cdots (2j_n)^{2j_n\beta} t_1^{2j_1} \cdots t_n^{2j_n} 
			 \\ & \le  (2k)^{2k\beta - k} L^{2k} \sum_{j_1+\ldots + j_n =k} c_{2j_1, \ldots ,2j_n} (2j_1)^{ j_1} \cdots (2j_n)^{j_n} t_1^{2j_1} \cdots t_n^{2j_n} 
			 \\ & \le (2k)^{2k\beta - k} (CL)^{2k}  \sum_{j_1+\ldots + j_n =k} c_{2j_1, \ldots ,2j_n} \Ex G_1^{2j_1} \cdots \Ex G_n^{2j_n} t_1^{2j_1} \cdots t_n^{2j_n}
			 \\ & = 
			  (2k)^{2k\beta - k} (CL)^{2k} \biggl\| \sum_{j=1}^n t_j G_j\biggr\|_{2k}^{2k} \le (2k)^{2k\beta}( CL)^{2k}\biggl\|\sum_{j=1}^n t_jY_{ij}\biggr\|_2^{2k} ,
		\end{align*}
		what finishes the proof of \eqref{remark comp of sums}.
\end{proof}
		
		By the claim we get
		\[
			\sigma_{p, cY_i}(q) \le CL q^\beta \sup_{s\in B_{p*}^n} \sqrt{ \sum_{j=1}^n s_j^2c_j^2} = CL q^\beta \maxj|c_j| \le CL^2 \min_{j\le n}\|Y_{ij}\|_q \maxj|c_j|,
		\]
		what allows us to obtain a version of \eqref{pom007} for $\varphi(t):=\min_{\substack{1\le i \le m,\\ 1\le j\le n}}  \|Y_{ij}\|_t.$
			\end{remark}
	
	\section{Estimates of norms of matrices in the case of Gaussian mixtures} \label{sect:Gmixtures}
	
	Let us recall the definition from \cite{MR3846841}, where  the significance of Gaussian mixtures is also described.
	
	\begin{definition}
		A random variable $X$ is called a (centred) Gaussian mixture if there exists a positive random variable $r$ and a standard Gaussian random variable $g$, independent of $r$, such that $X$ has the same distribution as  $rg$.
	\end{definition}
	
	We will work with  matrices of the form $(R_{ij}B_{ij}G_{ij})_{i\le m, j\le n}$ which entries  are Gaussian mixtures.  We additionally assume that $R_{ij}=|Z_{ij}|^\gamma$, where $\gamma\ge0$,  and that the matrix $Z$ is   log-concave and isotropic (considered as a random vector in $\er^{mn}$). It will be clear from the proof, that the corollary below is true also for another type of matrices: $(R_{i}B_{ij}G_{ij})_{i\le m, j\le n}$, where $R_i=|Z_i|^\gamma$, and $(Z_1,\ldots , Z_m)$ is an arbitrary isotropic log-concave random vector.
	
	\begin{corollary}\label{cor:Gmixtures}
	 	Let  $m,n\ge 2$, let $\gamma\ge 0$, let $B=(B_{ij})$ be a deterministic $m\times n $ matrix,  and let $G=(G_{ij})_{i\le m, j\le n}$ be a random matrix which entries are  i.i.d. standard Gaussian variables. Let $X_{ij} = |Z_{ij}|^\gamma B_{ij}G_{ij}$, where $Z=(Z_{ij})_{i\le m, j\le n}$ is a  log-concave and isotropic random matrix independent of $G$. Then for every $p,q \ge 2 \vee \frac 1\gamma$ we have
		\begin{multline*}
			\Ex \| X\|_{p',q}\le C(p,q,\gamma)\biggl[\bigl(\log m \bigr)^{\frac1q+\gamma} \maxi \bigl\| B_{i} \bigr\|_p 
			  + (\log n)^\gamma \maxj \bigl\| B^{(j)} \bigr\|_q \\
			  + 
			(\log m)^{1+\frac1q}\Ex \maxij |X_{ij}| \biggr] .
		\end{multline*}
	\end{corollary}
	
	\begin{proof}
		 Theorem \ref{main_thm_matrices} applied to  $Y=G$ and $A_{ij}=|Z_{ij}|^\gamma B_{ij}$ yields
		 \begin{multline*}
		 	\Ex \| X\|_{p',q}\le C(p,q)  \Bigl[ \bigl(\log m \bigr)^{1/q}  \Ex \maxi \bigl\| (B_{ij}|Z_{ij}|^\gamma)_j \bigr\|_p + \Ex\maxj \bigl\| (B_{ij}|Z_{ij}|^\gamma)_i \bigr\|_q 
			\\
			+ 
			\bigl(\log m \bigr)^{1+\frac1q}  \Ex \maxij |X_{ij}| \Bigr],
		 \end{multline*}
		 so it suffices to prove that 
		 \begin{equation}\label{aim_Gmixtures}
		 	\Ex \maxi \bigl\| (B_{ij}|Z_{ij}|^\gamma)_j \bigr\|_p\le C(p,\gamma) (\log m)^\gamma  \maxi \bigl\| B_{i} \bigr\|_p 
		\end{equation}
		and 
		\[
			\Ex\maxj \bigl\| (B_{ij}|Z_{ij}|^\gamma)_i \bigr\|_q \le C(q,\gamma) (\log n)^\gamma \maxj \bigl\| B^{(j)} \bigr\|_q
		\] 
		for $p\ge 1\vee \frac 1\gamma$.
		By the  symmetry of assumptions we need only to show \eqref{aim_Gmixtures}.
		
		If $\gamma<1$, then
		\[
				\Ex \maxi \bigl\| (B_{ij}|Z_{ij}|^\gamma)_j \bigr\|_p = 	\Ex \maxi \bigl\| (|B_{ij}|^{1/\gamma}|Z_{ij}|)_j \bigr\|_{p\gamma}^{\gamma} \le 	\Bigl(\Ex \maxi \bigl\| (|B_{ij}|^{1/\gamma}|Z_{ij}|)_j \bigr\|_{p\gamma} \Bigr)^{\gamma},
		\]
		and
		\[
			\bigl\| |B_{i}|^{1/\gamma} \bigr\|_{p\gamma}^\gamma =\bigl\| B_{i} \bigr\|_p,
		\]
		so it suffices to consider only $\gamma \ge 1$ (we used here the assumption that $p\ge \frac 1\gamma$).
		
		Note that for any $u\ge 1$ we have
		\begin{align}\label{pom1_Gmixtures}\nonumber
			\Ex \maxi \bigl\| (B_{ij}|Z_{ij}|^\gamma)_j \bigr\|_p  &= \Ex \maxi \bigl\| (|B_{ij}|^{1/\gamma} Z_{ij})_j \bigr\|_{p\gamma }^\gamma 
			\\ & \nonumber \le
			\Bigl( \Ex \maxi \bigl\| (|B_{ij}|^{1/\gamma} Z_{ij})_j \bigr\|_{p\gamma }^{u\gamma} \Bigr)^{1/u}
			\\ & \nonumber
			\le \Bigl( \Ex \sum_{i=1}^m \bigl\| (|B_{ij}|^{1/\gamma} Z_{ij})_j \bigr\|_{p\gamma}^{u\gamma} \Bigr)^{1/u}
			\\ & \le m^{1/u} \maxi  \Bigl(\Ex \bigl\| (|B_{ij}|^{1/\gamma} Z_{ij})_j \bigr\|_{p\gamma }^{u\gamma} \Bigr)^{1/u}.
		\end{align}
		
		Fix $i\le m$. By  Theorem \ref{thm:Paourlr} applied to $p=p\gamma $, $q=u\gamma$ (recall that $\gamma\ge 1$, so $u\gamma, p\gamma \ge 1$), and $Z_j=|B_{ij}|^{1/\gamma}Z_{ij}$ we have 
		\begin{align}\label{pom2_Gmixtures} \nonumber
			(Cp\gamma)^{-\gamma} \Bigl(\Ex \bigl\| (|B_{ij}|^{1/\gamma} Z_{ij})_j  \bigr\|_{p\gamma }^{u\gamma} & \Bigr)^{1/u} \le \Biggl[ \Ex \bigl\| (|B_{ij}|^{1/\gamma} Z_{ij})_j \bigr\|_{p\gamma }+ \sup_{t\in B_{p'}^n} \biggl\| \sum_{j=1}^n |B_{ij}|^{1/\gamma}Z_{ij}t_j \biggr\|_{u\gamma} \Biggr]^\gamma
			\\ & \le 
			2^{\gamma -1} \Biggl[ \Ex \bigl\| (|B_{ij}|^{1/\gamma} Z_{ij})_j \bigr\|_{p\gamma }^\gamma+ \sup_{t\in B_{p'}^n} \biggl\| \sum_{j=1}^n |B_{ij}|^{1/\gamma}Z_{ij}t_j \biggr\|_{u\gamma}^\gamma\Biggr].
		\end{align}
		Let us  use \eqref{regular_moments} and the assumption $\Ex Z_{ij}^2=1$ to
 estimate the first term in \eqref{pom2_Gmixtures}:		
 		\begin{equation}\label{pom3_Gmixtures}
			 \Ex \Bigl( \sum_{j=1}^n |B_{ij}|^p|Z_{ij}|^{p\gamma}\Bigr)^{1/p} \le \Bigl( \sum_{j=1}^n |B_{ij}|^p \Ex |Z_{ij}|^{p\gamma}\Bigr)^{1/p} \le (\Cr{seminorms} p\gamma)^\gamma \|B_i\|_p.
		\end{equation}
		Recall that $B_{p'}^n\subset B_2^n$. We use again  \eqref{regular_moments} and the isotropicity of $Z_i$ to estimate the second term in  \eqref{pom2_Gmixtures}:
		\begin{align}\label{pom4_Gmixtures} \nonumber
			\sup_{t\in B_{p'}^n} \biggl\| \sum_{j=1}^n |B_{ij}|^{1/\gamma}Z_{ij}t_j \biggr\|_{u\gamma}^\gamma& \le( \Cr{seminorms} u\gamma)^{\gamma} \sup_{t\in B_{2}^n} \biggl\| \sum_{j=1}^n  |B_{ij}|^{1/\gamma}Z_{ij} t_j \biggr\|_2^\gamma
			\\ &\nonumber
			= ( \Cr{seminorms} u\gamma)^{\gamma}  \sup_{t\in B_{2}^n} \biggl( \sum_{j=1}^n |B_{ij}|^{2/\gamma}t_j^2 \biggr)^{\gamma/2} 
			\\ & 
			= ( \Cr{seminorms} u\gamma)^{\gamma} \maxj |B_{ij}| \le ( \Cr{seminorms} u\gamma)^\gamma \|B_i\|_p.
		\end{align}
		
		Take $u=\log m$ and put together \eqref{pom1_Gmixtures}, \eqref{pom2_Gmixtures}, and \eqref{pom3_Gmixtures} to get the assertion.
	\end{proof}
	
	\begin{remark}\label{rem_pom_GM}
		Using \cite[Theorem 1.1]{MR3645120} instead of Theorem \ref{main_thm_matrices} in the proof above yields a slightly better estimate:
		\begin{multline*}
				\Ex \| X\|_{p',q}\le C(p,q) \biggl[ \bigl(\log m \bigr)^{\frac1q+\gamma} \maxi \bigl\| B_{i} \bigr\|_p 
			+ (\log n)^{\gamma} \maxj \bigl\| B^{(j)} \bigr\|_q \\ + 
			(\log m)^{1/q}\Ex \maxij |X_{ij}| \biggr] .	
		\end{multline*}
	\end{remark}
	
	\begin{remark}
		It is clear from the proof of Corollary \ref{cor:Gmixtures} that in the case  $Z_{ij}=G_{ij}'$, where $G_{ij}'$ are i.i.d. standard Gaussian variables, inequality \eqref{aim_Gmixtures} may be slightly improved:
		 \begin{equation}\label{pom_uncond}
		 	\Ex \maxi \bigl\| (B_{ij}|G_{ij}'|^\gamma)_j \bigr\|_p\le C(p,\gamma) (\log m)^{\gamma/2}  \maxi \bigl\| B_{i} \bigr\|_p 
		\end{equation}
		In order to obtain this improvement one should use $\|\langle t,G_i \rangle\|_{u\gamma} \lesssim \sqrt {u\gamma} \|\langle t,G_i \rangle\|_2$ instead of  $\|\langle t,Z_i \rangle\|_{u\gamma} \lesssim u\gamma \|\langle t,Z_i \rangle\|_2$. Therefore, if we additionally use Remark \ref{rem_pom_GM}, the assertion of Corollary \ref{cor:Gmixtures} in the case  $Z_{ij}=G_{ij}'$ (where $G'$ is independent of $G$) will state that
		\begin{multline}\label{eq_Gmixtures_gaussian}
			\Ex \| X\|_{p',q}\le C(p,q,\gamma)\biggl[\bigl(\log m \bigr)^{\frac1q+\frac{\gamma}2} \maxi \bigl\| B_{i} \bigr\|_p 
			  + (\log n)^{\gamma/2} \maxj \bigl\| B^{(j)} \bigr\|_q \\
			  + 
			(\log m)^{1/q}\Ex \maxij |X_{ij}| \biggr] .
		\end{multline}
	\end{remark}
	
	\begin{proof}[Proof of Proposition \ref{prop_matrices}]
		We begin similarly as in the proof of \eqref{aim_Gmixtures} (in the case $\gamma=1$), but we  estimate the second term on the right-hand side of \eqref{pom2_Gmixtures} in a slightly different way, using \eqref{regular_moments}:
		\begin{align*}
			\sup_{t\in B_{p'}^n} \biggl\| \sum_{j=1}^n B_{ij}Y_{ij}t_j \biggr\|_u \le n^{1/u} \sup_{t\in B_{p'}^n} \bigl( \Ex \maxj |t_jB_{ij}Y_{ij}|^u\bigr)^{1/u}\le n^{1/u} \Cr{seminorms} u \Ex \maxj|B_{ij}Y_{ij}|.
		\end{align*}
		We take $u=\log(m\vee n)$ to get the assertion.
	\end{proof}
	
	We may use the result concerning Gaussian mixtures to obtain the estimate similar to the one from Remark \ref{rem_p^gamma},  valid for all $\beta\ge \frac 12$ (not only for $\beta\in[\frac 12, 1]$), but with a slightly worse constants than in Remark \ref{rem_p^gamma}. The proof is based on the fact, that variables $Y_{ij}$ satisfying the moment assumption from Remark \ref{rem_p^gamma} are comparable with a certain Gaussian mixtures.
	
	\begin{corollary} \label{ex_GM}
		Let $m,n\ge 2$, $\gamma\ge \frac 12$, and let $X$ be an $m\times n$ random matrix with entries $X_{ij}=A_{ij}Y_{ij}$, where $Y_{ij}$ are independent symmetric random variables such that $\Ex Y_{ij}^2=1$. Assume that for any $r\ge 2$ and any $1\le i\le m$, $1\le j \le n$ we have $\frac{r^{\beta}}L\le\|Y_{ij}\|_r\le L r^{\beta}$. Then for all $p,q\ge 2$, 
		\begin{multline*}
		\Ex \| X\|_{p',q}\le C(p,q,L,\beta)\biggl[\bigl(\log m \bigr)^{\beta+\frac1q} \maxi \bigl\| A_{i} \bigr\|_p 
			  + (\log n)^{\beta} \maxj \bigl\| A^{(j)} \bigr\|_q \\
			  + 
			(\log m)^{1/q}\sqrt{\log (mn)}\Ex \maxij |X_{ij}| \biggr] .
		\end{multline*}
	\end{corollary}
	\begin{proof}
		Let $G_{ij}, G_{ij}'$, $i\le m$, $j\le n$, be i.i.d. standard Gaussian variables. Let $(\eps_{ij})$ be i.i.d. symmetric Bernoulli random variable, independent of $G$ and $G'$. Note that $Y_{ij}':=|G_{ij}|^{2\beta }\eps_{ij}$ satisfies $\frac{r^{\beta}}{L'}\le\|Y_{ij}'\|_r\le L' r^{\beta}$ for all $r\ge 2$, with a universal constant $L'$, since  $\|G_{ij}\|_s\sim \sqrt s$ for $s\ge 1$. Let $X' = (X_{ij})$ be the $m\times n$ random matrix with entries $X_{ij}' = A_{ij}Y_{ij}'$. By \cite[Lemma 4.7]{MR3878726} we know that 
		\[
			\frac 1{C(L,L', \beta)} \Ex  \vvvert X' \vvvert  \le\Ex \vvvert X \vvvert \le C(L,L', \beta) \Ex  \vvvert X' \vvvert 
		\] for any norm $\vvvert\cdot\vvvert$ on $m\times n$ real matrices. In particular
		\[
			\Ex \|X\|_{p',q}\le C(L,\beta) \Ex\|X'\|_{p',q}, \quad
		\text{ and } \quad
			\Ex \maxij |X_{ij}'| \le C(L,\beta)  \Ex \maxij |X_{ij}|.
		\]
		Moreover, by the Jensen inequality and by  \eqref{eq_Gmixtures_gaussian}  applied with $\gamma = 2\beta$ we have
		\begin{align*} 
			\Ex \bigl\|(X_{ij}') \bigr\|_{p',q} &  = \Ex \bigl\|(\eps_{ij}A_{ij}|G_{ij}'|^{2\beta}) \bigr\|_{p',q} = \sqrt{\frac \pi 2} \ \Ex \bigl\|\bigl(\Ex|G_{ij}|\eps_{ij}A_{ij}|G_{ij}'|^{2\beta}\bigr) \bigr\|_{p',q} 
			\\  & \le 
			 \sqrt{\frac \pi 2}  \ \Ex \bigl\|\bigl(|G_{ij}|\eps_{ij}A_{ij}|G_{ij}'|^{2\beta}\bigr) \bigr\|_{p',q}
			 =  \sqrt{\frac \pi 2}  \ \Ex_X \Ex_G \bigl\|(A_{ij}G_{ij}|G_{ij}'|^{2\beta}) \bigr\|_{p',q} 
			 \\ & \le C(p,q) \biggl(
			 	(\log m)^{\beta+\frac1q}\maxi \bigl\| A_{i} \bigr\|_p 
			  + (\log n)^{\beta} \maxj \bigl\| A^{(j)} \bigr\|_q 
			  \\ & \hspace{5,5cm}
			  + 
			(\log m)^{1/q}\Ex \maxij |A_{ij}G_{ij}|\cdot|G_{ij}'|^{2\beta} \biggr)
			 \\ & \le C(p,q) \biggl(
			 	(\log m)^{\beta+\frac1q}\maxi \bigl\| A_{i} \bigr\|_p 
			  + (\log n)^{\beta} \maxj \bigl\| A^{(j)} \bigr\|_q 
			  \\ & \hspace{5,3cm}
			  + 
			(\log m)^{1/q}\Ex \maxij |G_{ij}|\Ex \maxij |X_{ij}'|\biggr)  ,
		\end{align*}
		what yields the assertion, since $\Ex \maxij |G_{ij}| \sim \sqrt{\log(mn)}$.
			\end{proof}
	
	\section{The case of unconditional entries}\label{sect_uncond}
		\begin{proof}[Proof of Corollary \ref{thm_uncond_matrices}]
			Since $X$ is unconditional, it has the same distribution as the matrix $(\eps_{ij}X_{ij})_{i\le m, j\le n}$, where $\eps_{ij}$ are i.i.d. symmetric Bernoulli variables independent of $X$. Let $G_{ij}$ be i.i.d. standard Gaussian variables independent of $X$ and $(\eps_{ij})_{i\le m, j\le n}$. Then
		\begin{align*} 
			\Ex \bigl\|(X_{ij}) \bigr\|_{p',q} &  = \Ex \bigl\|(\eps_{ij}X_{ij}) \bigr\|_{p',q} = \sqrt{\frac \pi 2} \ \Ex \bigl\|\bigl(\eps_{ij}X_{ij}\Ex|G_{ij}|\bigr) \bigr\|_{p',q} 
			\\  & \le 
			 \sqrt{\frac \pi 2}  \ \Ex \bigl\|\bigl(\eps_{ij}X_{ij}|G_{ij}|\bigr) \bigr\|_{p',q}
			 =  \sqrt{\frac \pi 2}  \ \Ex_X \Ex_G \bigl\|(X_{ij}G_{ij}) \bigr\|_{p',q} 
			 \\ & \le C(p,q) \biggl(
			 	(\log m)^{1+\frac1q} \Ex_X\Ex_G \maxi \Bigl(\sum_{j=1}^n |X_{ij}G_{ij}|^p \Bigr)^{1/p}
				\\  &\hspace{5,9cm}+ \Ex_X\Ex_G  \maxj \Bigl(\sum_{i=1}^m |X_{ij}G_{ij}|^q \Bigr)^{1/q} \biggr) ,
		\end{align*}
		where in the last step we used Corollary \ref{cor_matrices} to estimate the mean  with respect to $G$. We use \eqref{pom_uncond} with $\gamma =1$ (to $\Ex_G$ in each term above separately) to get the assertion.
		\end{proof}
		
		\begin{remark}
			Using \cite[Theorem 1.1]{MR3645120} instead of Theorem \ref{main_thm_matrices} in the proof above yields a slightly better estimate in Theorem \ref{thm_uncond_matrices}:
			\begin{multline}
						\Ex  \| X\|_{p',q} \le
			 C(p,q) \biggl( (\log m)^{\frac12+\frac1q}\Ex \maxi \Bigl(\sum_{j=1}^n |X_{ij}|^p \Bigr)^{1/p} 
			 \\+ \sqrt{\log n}\Ex  \maxj \Bigl(\sum_{i=1}^m |X_{ij}|^q \Bigr)^{1/q}  \biggr).
		\end{multline}

		\end{remark}
	
	\section{Acknowledgements}
	I would like to thank Rafa{\l} Lata{\l}a for suggestions which helped me to make the presentation  clearer and more reader-friendly.

\bibliographystyle{amsplain}
\bibliography{references}

\end{document}